\documentclass[9pt]{amsart}
\usepackage{geometry}                
\usepackage{graphicx}
\usepackage{amssymb}
\usepackage{epstopdf}
\usepackage{array, tabularx}
\usepackage{color}
\usepackage{amsfonts,latexsym,rawfonts,amsmath,amsthm, mathrsfs, lscape}
\newtheorem{thm}{Theorem}[section]
\newtheorem{cor}[thm]{Corollary}
\newtheorem{lem}[thm]{Lemma}
\newtheorem{prop}[thm]{Proposition}

\theoremstyle{definition}
\newtheorem{defn}[thm]{Definition}
\theoremstyle{remark}
\newtheorem{rmk}[thm]{Remark}
\numberwithin{equation}{section}

\usepackage{color}
\definecolor{red}{rgb}{1,0,0}

\newcommand{\kibitz}[2]{\ifnum\Comments=1\textcolor{#1}{#2}\fi}

\title{The Continuity Method to Deform Cone Angle}
\author{Chengjian Yao}

\address{Department of Mathematics, Stony Brook University, Stony Brook, NY 11794}
\email{yao@math.sunysb.edu}
\begin{document}
\maketitle

\begin{abstract}
The continuity method is used to deform the cone angle of a weak conical K\"ahler-Einstein metric with cone singularities along a smooth anti-canonical divisor on a smooth Fano manifold. This leads to an alternative proof of Donaldson's Openness Theorem on deforming cone angle \cite{Don} by combining it with the regularity result of Guenancia-P$\breve{\text{a}}$un \cite{GP} and Chen-Wang \cite{CW}. This continuity method uses relatively less regularity of the metric (only weak conical K\"ahler-Einstein) and bypasses the difficult Banach space set up; it is also generalized to deform the cone angles of  a \emph{weak conical K\"ahler-Einstein metric} along a simple normal crossing divisor (pluri-anticanonical) on a smooth Fano manifold (assuming no tangential holomorphic vector fields). 
\end{abstract}
\section{Introduction}

In \cite{Don}, Donaldson defined a notion of \emph{conical K\"ahler-Einstein metric} where the K\"ahler potential is required to have $C^{2,\alpha,\beta}$ regularity and proved the Openness Theorem on deforming the cone angles by using a delicate version of the implicit function theorem depending on the Schauder estimate established therein.

 In this paper, smooth continuity paths are used to vary the cone angles of \emph{weak conical K\"ahler-Einstein metrics}, which are the solutions to the \emph{current equation}:

 \begin{equation}\label{current}
 \text{Ric }\omega=\beta\omega+2\pi(1-\beta)[D]
 \end{equation}
where $[D]$ is the \emph{current of integration} along $D$, and $\omega$ is a K\"ahler current on $X$ with bounded potential (in the sense of \cite{BBEGZ}). See Definition \ref{def2-1} for the precise definition.

\begin{thm}\label{thm1-1}
If there exists a weak conical K\"{a}hler-Einstein metric $\omega_{\varphi_\beta}$ with angle $2\pi \beta$ $(0<\beta<1)$ along a smooth anti-canonical divisor $D$ on a smooth Fano manifold $X$, then there exists $\delta>0$, such that for all $\beta'\in (\beta-\delta,\beta+\delta)$, there exists a weak conical K\"{a}hler-Einstein metric $\omega_{\varphi_{\beta'}}$ with angle $2\pi \beta'$ along the divisor $D$. 
\end{thm}

The method is a combination of two smooth continuity paths, the first one (the one that is used in \cite{CDS1}) is used to approximate $\omega_{\varphi_\beta}$ by smooth K\"ahler metrics with Ricci curvature bounded below by $\beta$, and the second one is used to deform the angle parameter $\beta$ to nearby $\beta'$. The potentials of the smooth approximating K\"ahler metrics are shown to have uniform $L^\infty$ bound, which enables us to pass to the higher order bounds. The weak conical K\"ahler-Einstein metric with angle $2\pi\beta'$ is obtained by taking the limit of these smooth K\"ahler metrics.

By the very recent regularity result of Guenancia-P$\breve{\text{a}}$un \cite{GP}, proved independently by Chen-Wang \cite{CW}, the deformed \emph{weak conical K\"ahler-Einstein metrics} are actually H\"older continuous, i.e. $\omega_{\varphi_{\beta'}}\in C^{2, \alpha, \beta'}$. Therefore, together with this regularity result, Theorem \ref{thm1-1} gives a new proof of Donaldson's Openness Theorem (Theorem 2 of \cite{Don}) for conical K\"ahler-Einstein metrics.


 Since the essential tool used in the study of continuity path, the \emph{log-Mabuchi-functional} (see Definition \ref{defn3-3}), requires relatively less regularity, we bypass the \emph{implicit function theorem} for singular K\"ahler metrics. This observation enables us to generalize the openness property to some other settings, where the Banach space theory seems difficult or subtle to set up. For instance, the case of \emph{simple normal crossing divisor (pluri-anticanonical)} on a smooth Fano manifold: 
\begin{thm}\label{thm1-2}
On a smooth Fano manifold $X$, suppose there is a \emph{weak conical K\"ahler-Einstein metric} with angle $2\pi\beta_i$ $(\text{for } \beta_i\in (0,1))$ along $D_i$, where $D_i\in |-\lambda_iK_X|$ with $\lambda_i>0$ and $\cup_i D_i$ being simple normal crossing, and also assume there is no holomorphic vector field tangential to any $D_i$. Then there exists $\delta>0$ small enough such that for all $\beta_i'\in (\beta_i-\delta,\beta_i+\delta)$, there exists a weak conical K\"ahler-Einstein metric with angle $2\pi\beta_i'$ along $D_i$ for $i=1, \cdots, k$ .
\end{thm}

\noindent $\mathbf{Acknowledgement}:$ The author is very grateful to his advisor, Professor Xiuxiong Chen, for introducing this problem about openness and giving insight on the possibility of bypassing the implicit function theorem for singular metrics. Great thanks also go to Dr.  Song Sun, Yuanqi Wang and Long Li for helpful discussions and useful suggestions.

\section{Setting Up Continuity Paths}\label{sect2}

Let $X$ be a smooth Fano manifold, fix a smooth background K\"ahler metric $\omega_0$ in $2\pi c_1(X)$, and write $h_{\omega_0}$ as its Ricci potential (with the normalization $\sup_X h_{\omega_0}=0$). Let $S$ be a holomorphic defining section of a smooth divisor $D\in |-K_{X}|$ and let $h$ be a smooth Hermitian metric on the line bundle $L_D$ defined by $D$ with curvature $\omega_0$. Let $[D]$ denote the \emph{current of integration} along $D$. We first define the notion of \emph{weak conical K\"ahler-Einstein metric}:

\begin{defn}[Weak Conical K\"ahler-Einstein Metric]\label{def2-1}
Let  $\omega_{\varphi_\beta}=\omega_0+\sqrt{-1}\partial\bar\partial \varphi_\beta$ for $\varphi_\beta\in C^\infty(X\backslash D)\cap C^0(X)$ be a smooth K\"ahler metric on $X\backslash D$ with bounded local K\"ahler potential near $D$. We say $\omega_{\varphi_\beta}$ is a \emph{weak conical K\"ahler-Einstein metric} with angle $2\pi\beta$ along $D$ if the following complex Monge-Amp$\grave{\text{e}}$re equation:

$$\omega_{\varphi_\beta}^n=e^{-\beta\varphi_\beta+h_{\omega_0}}\frac{\omega_0^n}{|S|_h^{2-2\beta}}.$$
is satisfied on $X\backslash D$, or equivalently the current equation 
$$\text{Ric }\omega_{\varphi_\beta}=\beta\omega_{\varphi_\beta}+2\pi(1-\beta)[D]$$
is satisfied on $X$.

\end{defn}

\begin{rmk}
The \emph{weak conical K\"ahler-Einstein metric} defined above is automatically quasi-isometric to a model conical K\"ahler metric near $D$ and has $C^{2,\alpha,\beta}$ regularity in the sense of Donaldson \cite{Don} by the regularity result of \cite{GP}. Also note that in \cite{Yao}, the smooth approximations obtained from the solutions of the smooth continuity paths are shown to have uniform $C^{1,1}$ bound. Therefore the \emph{weak conical K\"ahler-Einstein metric} has $C^{1,1}$ regularity, from which [CW] gives the $C^{2,\alpha,\beta}$ regularity. 
\end{rmk}

 By Poincar$\acute{\text{e}}$-Lelong's formula, $\chi=\omega_0+\sqrt{-1}\partial\bar{\partial} \log |S|_h^2=2\pi[D]$ as currents. We define a smooth approximation $\chi_\epsilon=\omega_0+\sqrt{-1}\partial\bar{\partial} \log (|S|_h^2+\epsilon)$ for $\epsilon\in (0,1]$.
 
Let $\omega_{\varphi_\beta}$ be a \emph{weak conical K\"ahler-Einstein metric}.  The authors of \cite{CDS1} used a smooth continuity path depending on two parameters to construct a family of smooth K\"ahler metrics (with Ricci curvature bounded below by $\beta$) to approximate $\omega_{\varphi_\beta}$ in the Gromov-Hausdorff sense. We adopt the same two parameter continuity path. Let us first recall the construction. Since the volume form $\omega_{\varphi_\beta}^n\in L^p(X, \omega_0^n)$ for $p\in [1,\frac{1}{1-\beta})$, we can choose a family of smooth volume forms $\eta_\epsilon$ with $\epsilon\in (0,1]$ approximating it in $L^p$, and then solve the Monge-Amp$\grave{\text{e}}$re equation 
\begin{equation}\label{eqn2-1}
\omega_{\varphi_\epsilon}^n=\eta_\epsilon
\end{equation}
by using the solution in \cite{Yau}.  Ko\l odziej's $L^p$ estimate \cite{Ko} asserts that $||\varphi_\epsilon||_{C^\gamma}$ is uniformly bounded for some $\gamma\in (0,1)$, and consequently $\varphi_\epsilon$ (normalized to have $0$ as supremum) subsequentially converges to some $\varphi_0$.  By a general uniqueness theorem  for bounded solutions of complex Monge-Amp$\grave{\text{e}}$re equations, $\varphi_0=\varphi_\beta+C$, and therefore we can assume $\varphi_\epsilon$ converges to $\varphi_\beta$ by passing to a subsequence.

Again from Yau's solution \cite{Yau}, we have smooth K\"ahler potentials $\psi_{\epsilon,\beta}$ such that 
\begin{equation}\label{eqn2-2}
\text{Ric }\omega_{\psi_{\epsilon,\beta}}=\beta \omega_{\varphi_\epsilon}+(1-\beta)\chi_\epsilon
\end{equation}
or equivalently satisfying Monge-Amp$\grave{\text{e}}$re equation

\begin{equation}\label{eqn2-3}
\omega_{\psi_{\epsilon,\beta}}^n=e^{-\beta\varphi_\epsilon+h_{\omega_0}}\frac{\omega_0^n}{(|S|_h^2+\epsilon)^{1-\beta}}
\end{equation}
It follows from Ko\l odziej's $L^p$ estimate \cite{Ko} that $||\psi_{\epsilon,\beta}||_{L^\infty(X)}$ is uniformly bounded.

Following \cite{CDS1} (equation 3.4), we then use the two parameter continuity path $\star_{\epsilon,t}^\beta$ with $\epsilon\in (0,1]$ and $t\in [0,\beta]$ to deform the K\"ahler metrics $\omega_{\psi_{\epsilon,\beta}}$:

\begin{equation*}
\star_{\epsilon,t}^\beta: \left\{  \begin{array}{ll}

          \text{Ric }\omega_{\phi_{\epsilon,t}^\beta}&=t \omega_{\phi_{\epsilon,t}^\beta}+(\beta-t)\omega_{\varphi_\epsilon}+(1-\beta)\chi_\epsilon \\
\phi_{\epsilon,0}^\beta&=\psi_{\epsilon,\beta}       
                   \end{array}\right. 
                   \end{equation*}

In \cite{CDS1}, the authors first use Donaldson's Openness Theorem to deform the cone angles of conical K\"ahler-Einstein metrics to show that the \emph{log-Mabuchi-functional} is coercive, which enables them to solve the above path for $t\in [0,\beta]$. However, in our situation, we can not use Donaldson's Openness Theorem to deform the cone angle and thus we lack of the ``coercivity'' of the \emph{log-Mabuchi-functional}. Instead, we use the \emph{modified log-Mabuchi-functional} (Definition \ref{defn3-3}) to establish the solvability for $t\in [0,\beta']$ for any parameter $\beta'<\beta$. Then we will argue by contradiction to solve $\star_{\epsilon,t}^\beta$ for $t\in [0,\beta]$ with a uniform $L^\infty$ bound on $\phi_{\epsilon,\beta}^\beta$. The next step is to start from $\phi_{\epsilon,\beta}^\beta$ and use the standard openness property for another smooth continuity path $\star_{\epsilon,t}$:

\begin{equation*}\star_{\epsilon,t}: \left\{ \begin{array}{ll}
    \text{Ric }\omega_{u_{\epsilon,t}}&=t\omega_{u_{\epsilon,t}}+(1-t)\chi_\epsilon\\
    u_{\epsilon,\beta}&=\phi_{\epsilon,\beta}^\beta
    \end{array} \right. 
    \end{equation*}

\noindent to deform the parameter $t$ to nearby $\beta'$ with uniform $L^\infty$ bound on $u_{\epsilon,\beta'}$. 
 
 \section{Solutions of $\star_{\epsilon,t}^\beta$ Up to $t=\beta$}
 
 First we need to recall some well-known functionals which are going to be used in this paper.
 
 \subsection{Functionals}
  
Fix $\omega_0$ to be a smooth K\"ahler metric on $X$.  The \emph{space of K\"ahler potentials} is 
$$\mathcal{H}=\{\phi\in C^{\infty}(X)|\omega_0+\sqrt{-1}\partial\bar\partial\phi>0\}$$
 

\begin{defn} \label{defn3-1}
 Let $\alpha\in 2\pi c_1(X)$ be a smooth closed $(1,1)$ form and $\chi$ be $2\pi[D]$ as before. For any $\phi\in \mathcal{H}$, take $\phi_t$ with $t\in [0,1]$ to be a smooth path in $\mathcal{H}$ connecting 
$0$ and $\phi$. All the functionals described below are defined by integration along the path $\phi_t$ ( it can be verified that the values do not depend on the particular choice of the path $\phi_t$).
\begin{equation}\label{}
 J_\alpha(\phi):=n\int_0^1 \mathrm{d}t \int_X \dot{\phi}_t (\alpha-\omega_{\phi_t})\wedge \omega_{\phi_t}^{n-1}
 \end{equation}
\begin{equation}\label{}
J_{\chi}(\phi)=2\pi n\int_{0}^1\mathrm{d}t\int_{D} \dot{\phi}_t\omega_{\phi_t}^{n-1}-n\int_0^1\mathrm{d}t\int_{X}\dot{\phi}_t\omega_{\phi_t}^{n}
\end{equation}
\text{[Mabuchi-functional]}
\begin{equation}\label{}
E(\phi)=-n\int_{0}^1\mathrm{d}t\int\dot{\phi}_t(\text{Ric } \omega_{\phi_t}-\omega_{\phi_t})\wedge\omega_{\phi_t}^{n-1}
\end{equation}
[log-Mabuchi-functional]
\begin{equation}\label{defn3-3}
E_{(1-\beta)D}(\phi)=E(\phi)+(1-\beta)J_\chi(\phi)
\end{equation}
\end{defn}

The functional $J_\alpha$ was introduced by Sz$\acute{\text{e}}$kelyhidi \cite{Sz} to study the Aubin-Yau continuity path initiated from different $(1,1)$ forms in $2\pi c_1(X)$. The \emph{Mabuchi-functional} was introduced by Mabuchi \cite{Ma} to study the K\"ahler-Einstein problem. The \emph{log-Mabuchi-functional} was used to study \emph{conical K\"ahler-Einstein metrics}, see \cite{Berm, BBEGZ, SW, LS, CDS1}. The notation here is the same as in \cite{CDS1}. In \cite{CDS1}, a smooth approximation with $\epsilon\in (0,1]$ of the \emph{log-Mabuchi-functional} was introduced to study the corresponding smooth continuity path:

\begin{equation}\label{}
E_{\epsilon,(1-\beta)D}(\phi):=E(\phi)+(1-\beta)J_{\chi_\epsilon}(\phi)
\end{equation}
where 
$$J_{\chi_\epsilon}(\phi)=n\int_{0}^1\mathrm{d}t\int_{X} \dot{\phi}_t(\chi_\epsilon-\omega_{\phi_t})\wedge\omega_{\phi_t}^{n-1}$$

Take $\alpha=\omega_0$, then
 $$J_{\omega_0}(\phi)=(I-J)(\phi)=\frac{1}{n+1}\sum_{i=0}^n\int_X \phi \omega_\phi^i\wedge\omega_0^{n-i}-\int_X \phi\omega_\phi^n$$ where $I$ and $J$ are the classical functionals defined by Aubin \cite{Aubin}, 
  $$I(\phi):=\int_X \phi\omega_0^n-\int_X\phi \omega_\phi^n;$$
 $$J(\phi):=\int_X \phi\omega_0^n-\frac{1}{n+1}\sum_{i=0}^n \int_X\phi \omega_\phi^i\wedge\omega_0^{n-i}$$
There is an easy comparison according to \cite{Aubin}:
 \begin{equation}\label{inequ3-6}
 \frac{1}{n+1}I(\phi)\leq J_{\omega_0}(\phi)\leq I(\phi)
 \end{equation}

The following proposition lists the properties needed in this paper.

\begin{prop}\label{prop3-3}
1. (\cite{Sz}) If $\alpha'=\alpha+\sqrt{-1}\partial \bar{\partial} \psi$, then 
 $$J_{\alpha'}(\phi)-J_{\alpha}(\phi)=\int_X \psi (\omega_{\phi}^n-\omega_0^n)$$
 2.(\cite{Chen}) The Mabuchi-functional has an well-known explicit formula:
 $$E(\phi)=\int_X \log \frac{\omega_\phi^n}{e^{h_{\omega_0}}\omega_0^n} \omega_\phi^n-J_{\omega_0}(\phi)+\int_X h_{\omega_0}\omega_0^n$$
 3.(\cite{Tian}) Let $\omega_0$ and $\omega_{\phi}=\omega_0+\sqrt{-1}\partial\bar\partial \phi$ be two smooth K\"ahler metrics on $X$, and let $C_P, C_S$ be the corresponding Poincar$\acute{\text{e}}$ and Sobolev constants. Then for some $\delta_n>0$, which is a dimensional constant, the following estimate holds:
 $$\text{Osc } \phi\leq \{C_S(\omega_0)^{\delta_n}C_P(\omega_0)+C_S(\omega_\phi)^{\delta_n}C_P(\omega_\phi)\} J_{\omega_0}(\phi)+C_S(\omega_0)^{\delta_n}+C_S(\omega_\phi)^{\delta_n}$$
 4.\cite{Chern, Lu}\label{CL}
Suppose $\omega$ and $\eta$ are two K\"{a}hler metrics on a compact K\"{a}hler manifold, if $\text{Ric }\omega\geq C_1\omega-C_2 \eta$ and the holomorphic bisectional curvature $R^\eta_{i\bar{j}k\bar{l}}\leq C_3(h_{i\bar{j}}h_{k\bar{l}}+h_{i\bar{l}}h_{k\bar{j}})$, then $$\Delta_\omega \log \text{tr}_\omega \eta \geq C_1-(C_2+2C_3)\text{tr}_\omega \eta.$$
 \end{prop}

As explained in the introduction, in order to solve $\star_{\epsilon,t}^\beta$, we need to introduce the \emph{modified log-Mabuchi-functional} by adding an extra term to achieve \emph{coercivity}. 

\begin{defn}{[modified log-Mabuchi-functional]}\label{defn3-3}
For any $\beta'<\beta$ and $\epsilon\in (0,1]$, define
\begin{equation*}
\widetilde{E}_{\epsilon,\beta'}:=E_{\epsilon,(1-\beta)D}+(\beta-\beta')J_{\omega_{\varphi_\epsilon}}
\end{equation*}
where $\varphi_\epsilon$ is the solution to equation \ref{eqn2-1}. 
\end{defn}

Notice that this modified functional depends on the choice of $\varphi_\epsilon$ and later we will see that its precise form is made such that it is decreasing along the continuity path $\star_{\epsilon,t}^\beta$ for $t\in [0,\beta']$. The following lemma shows that this \emph{modified log-Mabuchi-functional} is coercive.

\begin{prop}\label{lem3-4}
For any fixed $\beta'<\beta$, there exists $C=C_{\beta'}$ independent of $\epsilon$ such that for all $\epsilon\in (0,1]$:
$$\widetilde{E}_{\epsilon,\beta'}\geq (\beta-\beta')J_{\omega_0}-C$$
\end{prop}

\begin{proof} 
Rewrite the functional as following:
$$\widetilde{E}_{\epsilon,\beta'}=E+(1-\beta)J_{\chi_\epsilon}+(\beta-\beta')J_{\omega_{\varphi_\epsilon}}$$
$$=E_{(1-\beta)D}+(1-\beta)(J_{\chi_\epsilon}-J_\chi)+(\beta-\beta')(J_{\omega_{\varphi_\epsilon}}-J_{\omega_0})+(\beta-\beta')J_{\omega_0}.$$

\begin{itemize}

\item The first term $E_{(1-\beta)D}$ is the \emph{log-Mabuchi-functional}. It follows from \cite{Berm, BBEGZ} that $E_{(1-\beta)D}$ is bounded from below under the assumption that there exists a \emph{weak conical K\"ahler-Einstein metric} $\omega_{\varphi_\beta}$;
 \item The first error term $J_{\chi_\epsilon}-J_\chi $ is bounded from below by the following computation (see \cite{CDS1} for more detailed calculation):

\begin{align*}
(J_{\chi_\epsilon}-J_\chi)(\phi)&=\int_X \{\log (|S|_h^2+\epsilon)-\log |S|_h^2\}(\omega_\phi^n-\omega_0^n)\\
&\geq \int_X \{\log |S|_h^2-\log (|S|_h^2+\epsilon)\}\omega_0^n\\
&\geq -\sup_X \log (|S|_h^2+1)+\int_X \log |S|_h^2\omega_0^n
\end{align*}

\item The second error term $J_{\omega_{\varphi_\epsilon}}-J_{\omega_0}$, whose formula appears in the first part of Proposition \ref{prop3-3}, is bounded from below: 
\begin{align*}
(J_{\omega_{\varphi_\epsilon}}-J_{\omega_0})(\phi)&=\int_X \varphi_\epsilon(\omega_\phi^n-\omega_0^n)\\
&\geq -2 ||\varphi_\epsilon||_{L^\infty(X)}
\end{align*}
\end{itemize}

\noindent Therefore, if $C$ is chosen to be 
$$C=2(\beta-\beta')||\varphi_\epsilon||_{L^\infty(X)}-(1-\beta)\int_X \log |S|_h^2 \omega_0^n+(1-\beta)\sup_X \log(|S|_h^2+1),$$
the coercivity of $\widetilde{E}_{\epsilon,\beta'}$  in the lemma holds.
\end{proof}


In the next section, we will try to solve $\star_{\epsilon,t}^\beta$ for $t\in [0,\beta'],\epsilon\in (0,1]$ with a uniform $L^\infty$ bound on $\phi_{\epsilon,t}^\beta$. 

 \subsection{Solution for $\star_{\epsilon,t}^\beta$ with $t\in [0,\beta']$ and $\epsilon\in (0,1]$}
 
We need two lemmas which enable us to obtain uniform estimates  of $\phi_{\epsilon,t}^\beta$ along the continuity path $\star_{\epsilon,t}^\beta$. The first lemma appeared in the work of \cite{JMR} whose proof is based on the Chern-Lu inequality and is given here for completeness. 
\begin{lem}\cite{JMR}\label{lem3-6}
There exists $C=C_{A}$ such that if $\star_{\epsilon,t}^\beta$ have solution $\phi_{\epsilon,t}^\beta$ with
$\text{Osc }\phi_{\epsilon,t}^\beta\leq A$, then 
$$\omega_{\phi_{\epsilon,t}^\beta}\geq C^{-1}\omega_0$$
\end{lem}

\begin{proof}
By comparing the K\"ahler metric $\omega_{\phi_{\epsilon,t}^\beta}$ (denoted by $\omega_{\phi_\epsilon}$ for simplifying notation)  which has Ricci curvature bounded below by $0$ and the fixed smooth K\"ahler metric $\omega_0$ which has a fixed upper bound $\Lambda$ on the bisectional curvature, the Chern-Lu Inequality (see part $4$ of Proposition \ref{prop3-3})  tells that

$$\Delta_{\omega_{\phi_\epsilon}} \text{log } \text{tr}_{\omega_{\phi_\epsilon}}\omega_0 \geq -2\Lambda \text{tr}_{\omega_{\phi_\epsilon}}\omega_0.$$ 

\noindent By using the fact that $\Delta_{\omega_{\phi_\epsilon}}\phi_\epsilon=\text{tr}_{\omega_{\phi_\epsilon}} (\omega_{\phi_\epsilon}-\omega_0)=n-\text{tr}_{\omega_{\phi_\epsilon}}\omega_0$, we get the inequality

$$\Delta_{\omega_{\phi_\epsilon}} \{ \text{log } \text{tr}_{\omega_{\phi_\epsilon}} \omega_0 -(2\Lambda+1)\phi_\epsilon \} 
\geq \text{tr}_{\omega_{\phi_\epsilon}}\omega_0-n(2\Lambda+1).$$ 

\noindent The maximum principle on $X$ tells us $\text{tr}_{\omega_{\phi_\epsilon}}\omega_0 \leq \{n(2\Lambda+1)\} e^{(2\Lambda+1)\text{Osc }\phi_\epsilon}\leq C$. Thus the lower bound of $\omega_{\phi_\epsilon}$ claimed in this lemma is obtained.
\end{proof}
The second lemma is an application of Evans-Krylov's theorem \cite{Ev,Kr}  for $C^{2,\alpha}$ bound and higher order bounds.
\begin{lem}\label{lem3-7}
For any subset $K\subset\subset X\backslash D$, and $k\in \mathbb{N}$, there exists $C=C_{A,K,k}$ such that  if $\star_{\epsilon,t}^\beta$ have solutions $\phi_{\epsilon,t}^\beta$ with $\text{Osc }\phi_{\epsilon,t}^\beta\geq A$, then 
$$||\phi_{\epsilon,t}^\beta||_{C^k(K)}\leq C$$
\end{lem}
\begin{proof}
First we have the equation:
\begin{equation}
 \omega_{\phi_{\epsilon,t}^\beta}^n=e^{-t\phi_{\epsilon,t}^\beta-(\beta-t)\varphi_\epsilon+h_{\omega_0}}\frac{\omega_0^n}{(|S|_h^2+\epsilon)^{1-\beta}}
 \end{equation}
 By Lemma \ref{lem3-6} above, on any Euclidean ball $U$ inside $K$, the above equation reads as an equation of the type
 $$\omega_\phi^n=F\omega_0^n$$
 where $||F||_{C^\alpha}$ and $\Delta_{\omega_0}\phi$ are uniformly bounded. By the standard Evans-Krylov theorem \cite{Ev,Kr}, 
 $$||\phi||_{C^{2,\alpha}}\leq C$$ 
 and the higher order bound follows from a standard bootstrapping argument.
\end{proof}

\begin{prop}\label{prop3-7}
For any fixed $\beta'<\beta$, $\star_{\epsilon,t}^\beta$ is solvable for $t\in [0,\beta'],\epsilon\in (0,1]$  and there exists $C=C_{\beta'}$ such that 
 $$||\phi_{\epsilon,\beta'}^\beta||_{L^\infty(X)}\leq C$$ 
 \end{prop}

\begin{proof}
The proof follows a standard line, which is similar to the argument in \cite{CDS1} if the \emph{modified log-Mabuchi-functional} is used instead of the usual \emph{log-Mabuchi-functional}. For the readers' convenience, the calculation is included below.

 On the one hand, along the interval $t\in[0,\beta')$ on the continuity path $\star_{\epsilon,t}^\beta$,  the functional $\widetilde{E}_{\epsilon,\beta'}$ is decreasing by a direct calculation:

\begin{align*}
\frac{\mathrm{d}}{\mathrm{d}t} \widetilde{E}_{\epsilon,\beta'}(\phi_{\epsilon,t}^\beta)
=&-n\int_X \dot{\phi}_{\epsilon,t}^\beta\{\text{Ric }\omega_{\phi_{\epsilon,t}^\beta}-\omega_{\phi_{\epsilon,t}^\beta}\}\wedge \omega_{\phi_{\epsilon,t}^\beta}^{n-1}
+n(\beta-\beta')\int_X \dot{\phi}_{\epsilon,t}^\beta \{\omega_{\varphi_\epsilon}-\omega_{\phi_{\epsilon,t}^\beta}\}\wedge \omega_{\phi_{\epsilon,t}^\beta}^{n-1}\\
&+n(1-\beta)\int_X \dot{\phi}_{\epsilon,t}^\beta \{\chi_\epsilon-\omega_{\phi_{\epsilon,t}^\beta}\}\wedge \omega_{\phi_{\epsilon,t}^\beta}^{n-1}\\
=& n(\beta'-t)\int_X \dot{\phi}_{\epsilon,t}^\beta (\omega_{\phi_{\epsilon,t}^\beta}-\omega_{\varphi_\epsilon})\wedge \omega_{\phi_{\epsilon,t}^\beta}^{n-1}\\
=&(\beta'-t)\int_X (\phi_{\epsilon,t}^\beta-\varphi_\epsilon)\Delta_{\omega_{\phi_{\epsilon,t}^\beta}} \dot{\phi}_{\epsilon,t}^\beta \omega_{\phi_{\epsilon,t}^\beta}^n\\
=&(\beta'-t)\int_X (\phi_{\epsilon,t}^\beta-\varphi_\epsilon)\{-(\phi_{\epsilon,t}^\beta-\varphi_\epsilon)-t\dot{\phi}_{\epsilon,t}^\beta\}\omega_{\phi_{\epsilon,t}^\beta}^n\\
=&-(\beta'-t) \int_X (\phi_{\epsilon,t}^\beta-\varphi_\epsilon)^2 \omega_{\phi_{\epsilon,t}^\beta}^n
+t(\beta'-t)\int_X \dot{\phi}_{\epsilon,t}^\beta\{\Delta_{\omega_{\phi_{\epsilon,t}^\beta}}+t\}\dot{\phi}_{\epsilon,t}^\beta \omega_{\phi_{\epsilon,t}^\beta}^n\\
\leq & 0 
\end{align*}
where in the last inequality the second term is nonpositive by Lichnerowicz's estimate of the first eigenvalue of Laplacian operator on a manifold with a positive lower bound of the Ricci curvature.

On the other hand, since the initial potentials $\psi_{\epsilon,\beta}$ satisfes the Monge-Amp$\grave{\text{e}}$re equation \ref{eqn2-3}, by using the explicit formula for the \emph{log-Mabuchi-functional} in Proposition \ref{prop3-3}, we have

\begin{align*}
\widetilde{E}_{\epsilon,\beta'}(\psi_{\epsilon,\beta})&=\int_X \{-\beta \varphi_\epsilon-(1-\beta)\log (|S|_h^2+\epsilon)\}e^{-\beta\varphi_\epsilon+h_{\omega_0}}\frac{\omega_0^n}{(|S|_h^2+\epsilon)^{1-\beta}}-J_{\omega_0}(\psi_{\epsilon,\beta})+\int_X h_{\omega_0}\omega_0^n\\
&+(\beta-\beta')J_{\omega_{\varphi_\epsilon}}(\psi_{\epsilon,\beta})+(1-\beta)J_{\chi_\epsilon}(\psi_{\epsilon,\beta})
\end{align*}

\noindent Since $||\varphi_\epsilon||_{L^\infty(X)}$ and $||\psi_{\epsilon,\beta}||_{L^\infty(X)}$ are uniformly bounded from above by $C$, the first term 
\begin{align*}
|\int_X \{-\beta &\varphi_\epsilon-(1-\beta)\log (|S|_h^2+\epsilon)\}e^{-\beta\varphi_\epsilon
+h_{\omega_0}}\frac{\omega_0^n}{(|S|_h^2+\epsilon)^{1-\beta}}|\\
&\leq C\int_X|S|_h^{2\beta-2} |\log |S|_h^{2\beta-2}| \omega_0^n\leq C
\end{align*}
and the other terms

\begin{align*}
&J_{\omega_0}(\psi_{\epsilon,\beta}\leq I(\psi_{\epsilon,\beta})=\int_X \psi_{\epsilon,\beta}(\omega_0^n- \omega_{\psi_{\epsilon,\beta}}^n)\\
&J_{\omega_{\varphi_\epsilon}}(\psi_{\epsilon,\beta})=J_{\omega_0}(\psi_{\epsilon,\beta})+\int_X \varphi_\epsilon (\omega_{\psi_{\epsilon,\beta}}^n-\omega_0^n)\\
&J_{\chi_\epsilon}(\psi_{\epsilon,\beta})=J_{\omega_0}(\psi_{\epsilon,\beta})+\int_X \log(|S|_h^2+\epsilon)(\omega_{\psi_{\epsilon,\beta}}^n-\omega_0^n)
\end{align*}

\noindent are all bounded from above (independent of $\epsilon$). Thus the values of $\widetilde{E}_{\epsilon,\beta'}$ at the initial points $\psi_{\epsilon,\beta}$ are uniformly bounded (independent of $\epsilon$) from above.

The \emph{modified log-Mabuchi-functional} is coercive (see Lemma \ref{lem3-4}):
$$\widetilde{E}_{\epsilon,\beta'}(\phi)\geq (\beta-\beta')J_{\omega_0}(\phi)-C$$
Therefore along the interval $t\in[0,\gamma]$ where $\gamma\leq \beta'$ on which $\star_{\epsilon,t}^\beta$ can be solved,
$$J_{\omega_0}(\phi_{\epsilon,t}^\beta)\leq C$$
 for $C$ independent of $\epsilon\in (0,1]$ and $t\in [0,\gamma]$. Along the continuity path $\star_{\epsilon,t}^\beta$ for $t\in [\delta, \gamma]$, for some fixed $\delta>0$, 
 we have $\text{Ric }\omega_{\phi_{\epsilon,t}^\beta}\geq \delta\omega_{\phi_{\epsilon,t}^\beta}$, therefore the Poincar$\acute{\text{e}}$ and Sobolev constants of the family of K\"ahler metrics $\omega_{\phi_{\epsilon,t}^\beta}$ are uniformly bounded. By the third part of Propositon \ref{prop3-3}, 
 $$\text{Osc }\phi_{\epsilon,t}^\beta\leq C$$
 Since they satisfy the Monge-Amp$\grave{\text{e}}$re equation:
 
 \begin{equation}\label{eqn3-6}
 \omega_{\phi_{\epsilon,t}^\beta}^n=e^{-t\phi_{\epsilon,t}^\beta-(\beta-t)\varphi_\epsilon+h_{\omega_0}}\frac{\omega_0^n}{(|S|_h^2+\epsilon)^{1-\beta}}
 \end{equation}
 we can easily deduce that
 $$||\phi_{\epsilon,t}^\beta||_{L^\infty(X)}\leq C$$
Lemma \ref{lem3-6} and \ref{lem3-7} give us the higher order estimates for $\phi_{\epsilon,t}^\beta$ (which may depend on $\epsilon$), therefore the equation $\star_{\epsilon,t}^\beta$ can be solved up to $t=\beta'$ with uniform $L^\infty$ bound on $\phi_{\epsilon,t}^\beta$ (independent of $\epsilon$).
\end{proof}


Since $\beta'$ is any parameter smaller than $\beta$, and the continuity path $\star_{\epsilon,t}^\beta$ does not depend on $\beta'$, we immediately get the following corollary:
\begin{cor}\label{cor3-9}
For all $\epsilon\in (0,1]$, the continuity path $\star_{\epsilon,t}^\beta$ 
can be solved for $t\in [0,\beta)$.
\end{cor}

\subsection{Solution up to $t=\beta$ for $\star_{\epsilon,t}^\beta$}


We first need a simple convergence property about the $I$ functional that will be used several times later.
\begin{lem}\label{lem3-9}
If $\phi_{\epsilon_j,t_j}^\beta$ converges to $\phi$ in the $C^\alpha$ sense globally on $X$,  then 
$$\lim_{j\to\infty} I(\phi_{\epsilon_j,t_j}^\beta)=I(\phi).$$
\end{lem}

\begin{proof}
\noindent We have the formula  
  $$I(\phi_{\epsilon_j,t_j}^\beta)=\int_X \phi_{\epsilon_j,t_j}^\beta\omega_0^n
  -\int_X\phi_{\epsilon_j,t_j}^\beta e^{-t_j\phi_{\epsilon_j,t_j}^\beta-(\beta-t_j)\varphi_{\epsilon_j}
  +h_{\omega_0}}\frac{\omega_0^n}{(|S|_h^2+\epsilon_j)^{1-\beta}}
$$
and the second integrand is bounded by some $L^1$ function on $X$. The convergence claimed in the lemma follows from the Dominated Convergence Theorem.




 \end{proof}
 
 Since Corollary \ref{cor3-9} already assures that $\star_{\epsilon,t}^\beta$ can be solved on the open interval $[0,\beta)$ with uniform $L^\infty$ bound on $\phi_{\epsilon,\beta'}^\beta$ for any fixed $\beta'<\beta$, what remains to be shown is a bound on $J_{\omega_0}(\phi_{\epsilon,t}^\beta)$, uniform in $\epsilon$ and $t\in (\beta',\beta)$. Since Aubin's $I$ functional is equivalent to $J_{\omega_0}$ by the inequality \ref{inequ3-6}, we have
\begin{equation}\label{eqn3-9}
I(\phi_{\epsilon,\beta'}^\beta)\leq C
\end{equation}
for $C$ independent of $\epsilon$.
By a contradiction argument based on Berndtsson's Generalized Bando-Mabuchi Theorem \cite{Bern}, the required uniform bound is achievable by the following proposition:

\begin{prop}\label{prop3-10}

For any fixed $\beta'<\beta$, there exists $\epsilon_0=\epsilon_0(\beta')>0,$ such that for all $\epsilon\in (0,\epsilon_0]$  we have 
$$\sup_{t\in [\beta',\beta)} \{I(\phi_{\epsilon,t}^\beta)-I(\phi_{\epsilon,\beta'}^\beta)\} \leq 1.$$
\end{prop}

\begin{proof}
Argue by contradiction. Suppose the claimed estimate does not hold. Then there exists a sequence $\epsilon_j\searrow 0$,  with
 $$\sup_{t\in [\beta', \beta)} \{I(\phi_{\epsilon_j,t}^\beta)-I(\phi_{\epsilon_j,\beta'}^\beta) \}> 1$$ 
 Let $t_j$ be the first number $t \in (\beta',\beta)$ such that 
 $$I(\phi_{\epsilon_j,t_j}^\beta)-I(\phi_{\epsilon_j,\beta'}^\beta)=1$$

\noindent By the inequality \ref{eqn3-9},
$$I(\phi_{\epsilon_j,t_j}^\beta)\leq C+1$$
Then we get the uniform bound on $\phi_{\epsilon_j,t_j}^\beta$ and argue similarly 
 as in the proof of Proposition \ref{prop3-7}: we have all the higher order bounds
 $$||\phi_{\epsilon_j,t_j}^\beta||_{C^k(K)}\leq C_{k,K}$$
 on any $K\subset\subset X\backslash D$ and $k\in \mathbb{N}$. By taking a subsequence, we can assume that $\phi_{\epsilon_j,t_j}^\beta$ 
 converges to the $\phi$ in $C^\alpha$ sense globally on $X$ and in the $C^\infty$ sense away from $D$.


If $t_j\to t_\infty$, $\phi$ will be a solution to the current equation:

$$ \text{Ric }\omega_{\phi}=t_\infty \omega_{\phi}+(\beta-t_\infty)\omega_{\varphi_\beta}+(1-\beta)\chi$$

On the other hand, the \emph{weak conical K\"ahler-Einstein metric} $\omega_{\varphi_\beta}$ is also a solution:

$$ \text{Ric }\omega_{\varphi_\beta}=t_\infty \omega_{\varphi_\beta}+(\beta-t_\infty)\omega_{\varphi_\beta}+(1-\beta)\chi$$

Let $\Omega_1=t_\infty\omega_{\phi}$, $\Omega_2=t_\infty\omega_{\varphi_\beta}$ and $\theta=(\beta-t_\infty)\omega_{\varphi_\beta}+(1-\beta)\chi$.  Then $\Omega_1, \Omega_2$ give two bounded solutions to the equation:

$$\text{Ric }\Omega=\Omega+\theta$$ in the fixed cohomology class $2\pi t_\infty c_1(X)$. Berndtsson's Uniqueness Theorem \cite{Bern} implies that there exists $f\in Aut(X)$ which is generated by some holomorphic vector field $\mathbf{V}$ on $X$ such that 
$$f^*\Omega_1=\Omega_2$$
and
$$f^*\theta=\theta$$

Since $\theta=(\beta-t_\infty)\omega_{\varphi_\beta}+(1-\beta)\chi$ is the Siu's decomposition of a current, and $(\beta-t_\infty)\omega_{\varphi_\beta}$ has zero \emph{Lelong's number},
$$f^*\chi=\chi$$ 
Therefore $\mathbf{V}$ is tangential to $D$. However, there are no holomorphic vector fields tangential to $D$ according to Theorem 1.5 of Berman \cite{Berm} (by proving the \emph{properness} of the \emph{log-Mabuchi-functional}) or Theorem 2.1 of Song-Wang \cite{SW} (by pure algebraic geometry). Hence $f=id$, which implies that $\Omega_1=\Omega_2$, and in particular $\phi=\varphi_\beta$. And similarly $\phi_{\epsilon_j,\beta'}^\beta$ converges to $\varphi_\beta$ in the $C^\alpha$ sense globally on $X$.

However, by Lemma \ref{lem3-9} 
$$0=I(\varphi_\beta)-I(\varphi_\beta)=\lim_{j\to \infty}I(\phi_{\epsilon_j,t_j}^\beta)-I(\phi_{\epsilon_j,\beta'}^\beta)=1$$
which is a contradiction.
\end{proof}




With the uniform upper bound on the $I$ functional of Proposition \ref{prop3-10}, the same strategy as in Proposition \ref{prop3-7} would show the following:

\begin{prop}\label{prop3-11}
 There exists $\epsilon_0>0$ such that for all $\epsilon\in (0,\epsilon_0]$, the continuity path $\star_{\epsilon,t}^\beta$ is solvable up to $t=\beta$ with a uniform bound  
 $||\phi_{\epsilon,\beta}^\beta||_{L^\infty(X)}$
 for $\epsilon\in (0,\epsilon_0]$.
\end{prop}


\section{Deforming The Cone Angle from $\beta$}\label{sect4}

The endpoint $\omega_{\phi_{\epsilon,\beta}^\beta}$ of the continuity path $\star_{\epsilon,t}^\beta$ will be the starting point of the new two parameter family continuity path $\star_{\epsilon,t}$ defined in section \ref{sect2}, with $\epsilon\in (0,\epsilon_0]$:

\begin{equation*} \star_{\epsilon,t}:\left\{ \begin{array}{ll}
    \text{Ric }\omega_{u_{\epsilon,t}}&=t\omega_{u_{\epsilon,t}}+(1-t)\chi_\epsilon\\
    u_{\epsilon,\beta}&=\phi_{\epsilon,\beta}^\beta
    \end{array} \right. 
    \end{equation*}
    
    
Since $\chi_\epsilon$ is a strictly positive $(1,1)$ form, the linearized operator at $t=\beta$, which equals to $\Delta_{\omega_{u_{\epsilon,\beta}}}+\beta$, is invertible for some standard suitable Banach spaces. The standard \emph{implicit function theorem} enables us to perturb $t$ a little bit in both directions on $\star_{\epsilon,t}$ for $\epsilon\in (0,\epsilon_0]$. 

\begin{prop}\label{prop4-1} In both directions, there is uniform upper bound on the $I$ functional under small perturbation:
   \begin{itemize}
    \item There exists $\delta_1>0$ such that for all $\beta'\in (\beta-\delta_1,\beta)$, there exists $\epsilon_1\in (0,\epsilon_0]$ such that  for all
    $\epsilon\in (0,\epsilon_1)$
    $$\sup_{t\in (\beta',\beta]} \{I(u_{\epsilon,t})-I(u_{\epsilon,\beta})\}\leq 1$$
    \item There exists $\delta_2>0$ such that for all $\beta''\in (\beta,\beta+\delta_2)$, there exists $\epsilon_2\in (0,\epsilon_0]$ such that  for all
    $\epsilon\in (0,\epsilon_2)$
    $$\sup_{t\in [\beta,\beta'')} \{I(u_{\epsilon,t})-I(u_{\epsilon,\beta})\}\leq 1$$
    \end{itemize}
Consequently, take $\delta=\min(\delta_1,\delta_2)$ and $\underline\epsilon=\min(\epsilon_1,\epsilon_2)$, then
$u_{\epsilon,t}$ will have uniform $L^\infty$ bound for $\epsilon\in (0, \underline\epsilon]$ and $t\in (\beta-\delta,\beta+\delta)$.
\end{prop}

\begin{proof}
Since the two directions are similar, we only prove the right direction. The proof uses the same idea as in Proposition \ref{prop3-10}.
We argue by contradiction based on normalization of the $I$ functional. Assume the conclusion is not true, then we can find a sequence $\epsilon_j\searrow 0$ and $\beta_j''=\beta+\delta_j$ with $\delta_j\searrow 0$, such that
$$\sup_{t\in [\beta,\beta_j'']} \{I(u_{\epsilon_j,t})-I(u_{\epsilon_j,\beta})\}>1$$
Let $t_j$ be the first number bigger than $\beta$ such that
$$I(u_{\epsilon_j,t_j})-I(u_{\epsilon_j,\beta})=1$$
 Since $I(u_{\epsilon_j,\beta})$ is uniformly bounded by Proposition \ref{prop3-11},  
 $$I(u_{\epsilon_j,\beta})\leq C,$$
 which implies that 
 $$I(u_{\epsilon_j,t_j})\leq C+1$$
 The standard argument gives the $L^\infty$ bound since the Ricci curvature of this family of metrics is uniformly bounded below by some positive constant and the volume is a fixed topological constant. 
 
The continuity path $\star_{\epsilon,t}$ corresponds to the Monge-Amp$\grave{\text{e}}$re equation 
 
 \begin{equation}\label{eqn4-2}
 \omega_{u_{\epsilon,t}}^n=e^{-tu_{\epsilon,t}+h_{\omega_0}}\frac{\omega_0^n}{(|S|_h^2+\epsilon)^{1-t}}
 \end{equation}
 
\noindent Even though this equation is different from equation \ref{eqn3-6}, the uniform $L^\infty$ bound of $u_{\epsilon_j,t_j}$ will imply the R.H.S. is uniform in $L^p$ for some $p>1$, which yields the global $C^\alpha$ bound for some $\alpha\in (0,1)$ by \cite{Ko}.
 
Now by using the analogous result of Lemma \ref{lem3-6} and \ref{lem3-7} for equation \ref{eqn4-2}, we get all the higher order bounds of $u_{\epsilon_j,t_j}$ away from $D$. Then by taking the limit (subsequentially) $u_{\epsilon_j,t_j}$ converges to some $u$ in the $C^\alpha$ sense globally on $X$ and $C^\infty$ sense away from $D$, and $u_{\epsilon_j, \beta}$ converges to some $v$ in the same sense. We know that $u$ and $v$ are both solutions to the equation:

$$ \text{Ric }\omega=\beta \omega+(1-\beta)\chi$$

By the same argument as in Proposition \ref{prop3-10}, 
$$u=v=\varphi_\beta.$$ Similar to Lemma \ref{lem3-9}, we have the convergence of the functional $I$ by the Dominated Convergence Theorem:
$$0=I(u)-I(v)=\lim_{j\to \infty}I(u_{\epsilon_j, t_j})-I(u_{\epsilon_j, \beta})=1$$
Contradiction.
\end{proof}

As a consequence, we can finish the proof of Theorem \ref{thm1-1}.

\begin{proof}[\bfseries{Proof of Theorem \ref{thm1-1}}]
Proposition \ref{prop4-1} gives a family of solutions with $\epsilon \in (0,\underline\epsilon]$ and $\beta'\in (\beta-\delta,\beta+\delta)$:

\begin{equation}
\omega_{u_{\epsilon, \beta'}}^n=e^{-\beta' u_{\epsilon,\beta'}+h_{\omega_0}}\frac{\omega_0^n}{(|S|_h^2+\epsilon)^{1-\beta'}}
\end{equation}
with $||u_{\epsilon,\beta'}||_{L^\infty}\leq C$ independent of $\epsilon$ and $\beta'$. Then by taking a limit, $\varphi_{\beta'}=\lim_{\epsilon\to 0}u_{\epsilon,\beta'}$ is bounded on $X$ and smooth away from $D$. Moreover it satisfies:
$$\omega_{\varphi_{\beta'}}^n=e^{-\beta'\varphi_{\beta'}+h_{\omega_0}}\frac{\omega_0^n}{|S|_h^{2-2\beta}}.$$
Therefore $\omega_{\varphi_{\beta'}}$ is a \emph{weak conical K\"ahler-Einstein metric} of angle $2\pi\beta'$ along $D$.
\end{proof}

\begin{rmk}
As is shown in \cite{Yao}, around each point $p$ on $D$ a uniform $C^{1,1}$ comparison for the smooth approximation $\omega_{u_{\epsilon,\beta'}}$ holds, i.e.
$$C^{-1}\omega_{(\beta',\epsilon)}\leq \omega_{u_{\epsilon,\beta'}}\leq C\omega_{(\beta',\epsilon)}$$
where $\omega_{(\beta',\epsilon)}=\sqrt{-1}\{\beta'^2(|z_1|^2+\epsilon)^{\beta-1}\mathrm{d}z_1\wedge\mathrm{d}\bar{z}_1+\sum_{i=2}^n\mathrm{d}z_i\wedge\mathrm{d}\bar{z}_i\}$ for $D$ locally defined by $\{z_1=0\}$. A similar result was obtained by \cite{GP} with a different method.
\end{rmk}

\section{A More General Case--Simple Normal Crossing Divisor (Pluri-anticanonical)}
In this section, we use the same idea to deform the cone angles for \emph{weak conical K\"ahler-Eistein metrics} along a simple normal crossing divisor (pluri-anticanonical) on a smooth Fano manifold. 

Let $X$ be a smooth Fano manifold with a smooth background K\"ahler metric $\omega_0$, and let $D_1,D_2,\cdots,D_k$ be a collection of smooth hypersurfaces in $X$ with simple normal crossing at any intersection points. Let $\mathbf{b}=(\beta_1,\cdots,\beta_k)$ be a $k$-tuple of numbers in $(0,1)$. Let $S_i$ be the defining section of $D_i$, and choose a smooth Hermitian metric $h_i$ on the holomorphic line bundle $L_{D_i}$ with smooth curvature form $\alpha_i$. Let $\chi^i=\alpha_i+\sqrt{-1}\partial\bar\partial \log |S_i|_{h_i}^2=2\pi [D_i]$.
Generalizing Definition \ref{def2-1}, we can define \emph{weak conical K\"ahler-Einstein metrics} with angles $2\pi\beta_i$ along $D_i$ (see \cite{GP} for a more general definition for a KLT pair).

\begin{defn}\label{defn5-1}
 Let  $\omega_{\varphi_\mathbf{b}}$ be a smooth K\"ahler metric on $X\backslash \cup_i D_i$ with the potential $\varphi_\mathbf{b}$ bounded on $X$. If it satisfies:

\begin{equation}\label{eqn5-1}
\omega_{\varphi_{\mathbf{b}}}^n=e^{-\mu\varphi_{\mathbf{b}}+h_{\omega_0}}\frac{\omega_0^n}{\Pi_{i=1}^k |S_i|_{h_i}^{2-2\beta_i}}
\end{equation}
or equivalently the equation:

\begin{equation}
\text{Ric }\omega_{\varphi_{\mathbf{b}}}=\mu\omega_{\varphi_{\mathbf{b}}}+\sum_{i=1}^k (1-\beta_i)\chi^i
\end{equation}
it is called a \emph{weak conical K\"ahler-Einstein metric} on $X$ with angle $2\pi\beta_i$ along $D_i$. 
\end{defn}

 It was shown by Guenancia-P$\breve{\text{a}}$un \cite{GP} that the \emph{weak conical K\"ahler-Einstein metric} defined above has ``conical''  behavior in the sense that it is quasi-isometric to a standard local model conical K\"ahler metric near each point on $\cup_i D_i$ (Theorem $A$ of \cite{GP}) and the metric is actually H\"older continuous (Theorem $B$ of \cite{GP}) . 
 
 Notice that in this situation the \emph{cohomological condition} holds:
\begin{equation}\label{coho}
2\pi c_1(X)=\mu [\omega_{\varphi_{\mathbf{b}}}]+2\pi \sum_{i=1}^k (1-\beta_i)[D_i]
\end{equation}

\noindent {\bfseries{Proof of Theorem \ref{thm1-2}  }  }
If $D_i\in |-\lambda_i K_X|$ with $\lambda_i>0$, then $\alpha_i$ can be chosen to be $\lambda_i\omega_0$ by suitable choice of $h_i$. Let 
$$\chi_\epsilon^i=\lambda_i\omega_0+\sqrt{-1}\partial\bar\partial \log (|S_i|_{h_i}^2+\epsilon)$$
with $\epsilon\in (0,1]$ be the smoothing of 
$$\chi^i=\lambda_i\omega_0+\sqrt{-1}\partial\bar\partial \log |S_i|_{h_i}^2=2\pi [D_i]$$
Let $\mu=r(\mathbf{b}):=1-\sum_{i=1}^k \lambda_{i}(1-\beta_i)$.
The first step is also to approximate $\omega_{\varphi_\mathbf{b}}$ by smooth K\"ahler metrics $\omega_{\varphi_{\epsilon,\mathbf{b}}}$ in the $C^\gamma$ sense. For any $\mu'<\mu$, we introduce the analogous  \emph{modified log-Mabuchi-functional}

$$\widetilde{E}_{\epsilon,\mu'}=E+\sum_{i=1}^k(1-\beta_i)J_{\chi_\epsilon^i}+(\mu-\mu')J_{\omega_{\varphi_{\epsilon,\mathbf{b}}}}$$
where $J_{\chi_\epsilon^i}$ is defined by slightly generalizing the $J_{\alpha}$ functional defined in Definition \ref{defn3-1} to the case where $\alpha\in 2\pi \lambda c_1(X)$. 

This functional can be used to solve the continuity path $\star_{\epsilon,t}^\mathbf{b}$:

\begin{equation}\label{eqn5-3}
\text{Ric }\omega_{\phi_{\epsilon,t}^\mathbf{b}}=t\omega_{\phi_{\epsilon,t}^\mathbf{b}}
+(\mu-t)\omega_{\varphi_{\epsilon, \mathbf{b}}}+\sum_{i=1}^k (1-\beta_i)\chi_\epsilon^i
\end{equation}
up to $t=\mu$ with uniform $L^\infty$ bound on $\phi_{\epsilon,\mu}^\mathbf{b}$. 
One remark is that along the continuity path $\star_{\epsilon,t}^\mathbf{b}$, in the case of negative or zero Ricci curvature, the $L^\infty$ bound is obtained by the maximum principle or Ko\l odziej's estimate \cite{Ko}. Comparatively, in the case of positive Ricci curvature, we use Moser's Iteration, which depends on the uniform bound of the Sobolev constant and Poincar$\acute{\text{e}}$ constant (both hold on the continuity path). Another remark is that the condition \emph{simple normal crossing} is required since Ko\l odziej's estimate requires uniform $L^p$ bound of the R.H.S. of equation \ref{eqn5-1} some some $p>1$. And in the contradiction argument here we need to use the Uniqueness Theorem (Theorem $5.1$ of \cite{BBEGZ}) for weak K\"ahler-Einstein metrics.

We can then deform the angle parameters $\beta_i's$ as in section \ref{sect4}. The potentials $\phi_{\epsilon,\mu}^\mathbf{b}$ give the starting points at $t=\mu$ for the continuity paths:

 \begin{equation}\label{eqn5-4}
 \text{Ric }\omega_{u_{\epsilon,t}}=t\omega_{u_{\epsilon,t}}+\sum_{i=1}^k (1-\beta_{i,t})\chi_\epsilon^i
 \end{equation}
 
Since there are $k$ parameters $\beta_1,\cdots, \beta_k$ to vary and only one time parameter, we need to use successive continuity paths to achieve this, i.e. firstly deform $\beta_1$ to $\beta_1'$ nearby,  and then start from this new weak conical K\"ahler-Einstein metric to deform the second angle $\beta_2$ to $\beta_2'$ nearby, and so on $\cdots$. After $k$ steps, we can deform $(\beta_1,\beta_2, \cdots, \beta_k)$ to $(\beta_1',\beta_2',\cdots,\beta_k')$ nearby. This finishes the proof of Theorem \ref{thm1-2}.


\begin{thebibliography}{99}
\bibitem[Au]{Aubin}T. Aubin. \emph{R$\acute{\text{e}}$duction du cas positif de l'$\acute{\text{e}}$quation de Monge-Amp$\grave{\text{e}}$re sur les vari$\acute{\text{e}}$t$\acute{\text{e}}$ K\"ahleriennes 
compactes $\grave{\text{a}}$ la d$\acute{\text{e}}$monstration d'une in$\acute{\text{e}}$gualit$\acute{\text{e}}$}, J. Funct. Anal., 57 (1984), 143-153.
\bibitem[Berm]{Berm}R. Berman. \emph{A thermodynamical formalism for Monge-Amp$\grave{\text{e}}$re equations, Moser-Trudinger inequalities and K\"ahler-Einstein metrics}, arXiv:1011.3976.
\bibitem[Bern]{Bern}B. Berndtsson, \emph{A Brunn-Minkowski Type Inequality For Fano Manifolds and The Bando-Mabuchi Uniqueness Theorem}, arXiv. 1103.0923.
\bibitem[BBEGZ]{BBEGZ}R. J. Berman, S. Boucksom, P. Eyssidieux, V. Guedj, and A. Zeriahi, \emph{K\"ahler-Einstein metrics and the K\"ahler-Ricci flow on Log Fano varieties}. arXiv:1111.7158.
\bibitem[Chern]{Chern}S-S. Chern. \emph{On holomorphic mappings of Hermitian manifolds of the same dimension}, Proc. Symp. Pure Math. 11, American Mathematical Society, 1968, pp. 157Ð170.
\bibitem[Chen]{Chen}X-X. Chen. \emph{On the lower bound of the Mabuchi energy and its application}. Internat. Math. Res. Notices 2000, no. 12, 607-623.
\bibitem[CDS1]{CDS1}X-X. Chen, S. Donaldson and S. Sun, \emph{K\"{a}hler-Einstein metrics on Fano manifolds, I: approximation of metrics with cone singularities}, arXiv:1211.4506.
\bibitem[CGP]{CGP}F. Campana, H. Guenancia and M. P$\breve{\text{a}}$un, \emph{Metrics with cone singularities along normal crossing divisors and holomorphic tensor fields}, Annales scientifiques de l'ENS 46, fascicule 6 (2013), 879-916
\bibitem[CW]{CW}X-X. Chen, Y-Q. Wang, \emph{On the regularity problem of complex Monge-Amp$\grave{\text{e}}$re equations with conical singularities}, arXiv:1405.1201.
\bibitem[Don]{Don}S. K. Donaldson, \emph{K\"ahler metrics with cone singularities along a divisor}, arXiv: 1102.1196v2.
\bibitem[Ev]{Ev}L.C. Evans, \emph{Classical solutions of fully nonlinear, convex, second-order elliptic equations}, Comm. Pure Appl. Math. 35 (1982), 333Ð363.
\bibitem[EGZ]{EGZ}P. Eyssidieux, V. Guedj, A. Zeriahi, \emph{Singular K\"ahler-Einstein metrics}, arXiv.0603431.
\bibitem[GP]{GP}H. Guenancia,  M. P$\breve{\text{a}}$un, \emph{Conic singularities metrics with Prescribed Ricci curvature: The case of General cone angles along normal crossing divisors.} arXiv:1307.6375.
\bibitem[JMR]{JMR}T.D. Jeffres, R. Mazzeo and Y. Rubinstein, \emph{K\"{a}hler-Einstein metrics with Edge singularities (with an appendix by C. Li and Y.A. Rubinstein)}, preprint, 2011, arxiv:1105.5216. In revision for Annals of Math.
\bibitem[LS]{LS}C. Li, S. Sun, \emph{Conical K\"{a}hler Einstein metrics revisited}, Comm. Math. Phys. 331(2014), no. 3, 927-973.
\bibitem[Ko]{Ko}S. Ko\l odziej, \emph{H\"older continuity of solutions to the complex Monge-Amp$\grave{text{e}}$re equation with the right hand
side in $L^p$, the case of compact K\"ahler manifolds}, Math. Ann. 342 (2008), 379-386.
\bibitem[Kr]{Kr}N.V. Krylov, \emph{Boundedly inhomogeneous elliptic and parabolic equations (Russian)}, Izv. Akad. Nauk SSSR Ser. Mat. 46 (1982), 487Ð523.
\bibitem[Lu]{Lu}Y-C. Lu, \emph{Holomorphic Mapping of Complex Manifolds}, J. Diff. Geom. 2 (1968), 299Ð312.
\bibitem[Ma]{Ma}T. Mabuchi. \emph{K-energy maps integrating Futaki invariant}, T$\hat{\text{o}}$hoku Math. J. 38(1986) 575-593.
\bibitem[SW]{SW} J. Song and X-W. Wang, \emph{The Greatest Ricci Lower Bound, Conical Einstein Metrics and the Chern Number Inequality}, arXiv: 1207.4839.
\bibitem[Sz]{Sz}G. Sz$\acute{\text{e}}$kelyhidi, \emph{Greatest Lower Bounds on the Ricci Curvature of Fano Manifolds}, Compositio Math. 147(2011),319-331.
\bibitem[Tian]{Tian}G. Tian. \emph{K\"ahler-Einstein metrics with Positive Scalar curvature}. Invent. Math. 137(1997)1-37.
\bibitem[Yao]{Yao}C-J.Yao, \emph{Existence of Weak Conical K\"ahler-Einstein metrics along Smooth Hypersurfaces}, arXiv. 1308.4307.
\bibitem[Yau]{Yau}S-T. Yau, \emph{On the Ricci curvature of a compact K\"ahler manifold and the complex Monge-Amp$\grave{\text{e}}$re equation}, I. Comm. Pure Appl. Math., 31, 339±441 (1978)
\end{thebibliography}
\end{document}